\documentclass{amsart}
\usepackage{amsfonts, amsbsy, amsmath, amssymb}

\newtheorem{thm}{Theorem}[section]
\newtheorem{lem}[thm]{Lemma}

\numberwithin{equation}{section}

\theoremstyle{definition}

\newcommand\numberthis{\addtocounter{equation}{1}\tag{\theequation}}

\newcommand{\x}{{\tt x}}

\allowdisplaybreaks

\begin{document}

\title[a Type of Permutation Trinomials]{Determination of a Type of Permutation Trinomials over Finite Fields, II}

\author[Xiang-dong Hou]{Xiang-dong Hou}
\address{Department of Mathematics and Statistics,
University of South Florida, Tampa, FL 33620}
\email{xhou@usf.edu}
\thanks{* Research partially supported by NSA Grant H98230-12-1-0245.}



\begin{abstract}
Let $q$ be a prime power. We determine all permutation trinomials of $\Bbb F_{q^2}$ of the form $a\x+b\x^q+\x^{2q-1}\in\Bbb F_{q^2}[\x]$. The subclass of such permutation trinomials of $\Bbb F_{q^2}$ with $a,b\in\Bbb F_q$ was determined in a recent paper \cite{Hou-a}.
\end{abstract}

\maketitle

%%%%%%%%%%%%%%%%%%%%%%%%%%%%%%%%%%%%%%%%
%   section 1
%%%%%%%%%%%%%%%%%%%%%%%%%%%%%%%%%%%%%%%%

\section{Introduction}

For a prime power $q$, let $\Bbb F_q$ denote the finite field with $q$ elements. A polynomial $f\in\Bbb F_q[\x]$ is called a {\em permutation polynomial} (PP) of $\Bbb F_q$ if it induces a permutation of $\Bbb F_q$. In a recent paper \cite{Hou-a}, all permutation trinomials of $\Bbb F_{q^2}$ of the form $f=a\x+b\x^q+\x^{2q-1}\in\Bbb F_q[\x]$ were determined. (Note the assumption that $a,b\in\Bbb F_q$.) In fact, \cite{Hou-a} gives explicit conditions on $a,b\in\Bbb F_q$ that are necessary and sufficient for $f$ to be a PP of $\Bbb F_{q^2}$. The motivation for studying this type of trinomials was to solve a related problem about another class of PPs of finite fields; we refer to the reader to \cite{FHL,Hou-a} for the details.

In the present paper, we consider the same type of trinomials $f=a\x+b\x^q+\x^{2q-1}$ but with $a,b\in\Bbb F_{q^2}$. We find explicit conditions on $a$ and $b$ that are necessary and sufficient for $f$ to be a PP of $\Bbb F_{q^2}$. The results appear in Section 2 as Theorems A and B, which cover the odd $q$ and even $q$ cases separately. Thus we have a complete determination of permutation trinomials of $\Bbb F_{q^2}$ of the type $a\x+b\x^q+\x^{2q-1}$.

The proofs of Theorems A and B are given in Sections 3 and 4, respectively. The basic strategy and the general plan for the proofs in the present paper are the same as those of \cite{Hou-a}. The necessity of the conditions is obtained through the computation of the power sums $\sum_{x\in\Bbb F_{q^2}}f(x)^s$ for suitable $s$; the sufficiency of the conditions is proved by establishing the uniqueness of the solution $x$ in $\Bbb F_{q^2}$ of the equation $f(x)=y$ with $y\in\Bbb F_{q^2}$. However, certain parts of the proofs in \cite{Hou-a} do not work without the assumption that $a,b\in\Bbb F_q$. The proofs of the present paper rely on some new techniques which allow us to overcome the difficulties that cannot be resolved by simple adaptations of the method of \cite{Hou-a}. 
For example, a $p$-adic method is used to capture the information of the discriminant of a polynomial lost in characteristic $2$. Similar to the situation in \cite{Hou-a},
the proofs in the present paper are very much dependent on computations. A highlight of the computations is equations \eqref{3.19} and \eqref{3.20} where mysterious factorizations are found with computer assistance. For the computations that are similar to those of \cite{Hou-b,Hou-a}, we will be brief and we refer the reader to \cite{Hou-b,Hou-a} for the details.

There are numerous classes of permutation trinomials over finite fields in the literature, some of which have profound connections to other areas. We refer the reader to a recent survey \cite{Hou-Fq11} and the references therein. To the author's best knowledge, the results of the present paper are the only instance where a (nontrivial) type of permutation trinomials is completely determined without additional assumptions on the coefficients. It is our hope that the solution of the problem considered here can be a step stone leading to the determination of similar types of permutation trinomials. The trinomial $a\x+b\x^q+\x^{2q-1}$ can be expressed as $\x h(\x^{q-1})$, where $h=a+b\x+\x^2$. More generally, PPs of $\Bbb F_q$ of the form $\x^r h(\x^{(q-1)/d})$, where $r>0$ and $d\mid q-1$, have been the focus of several studies \cite{Akb-Wan07, Akb-Ghi-Wan11, Mar11, Par-Lee01,Wan-Lid91,Wan07,Zie08,Zie09,Zie-arXiv}. It is also our hope that the present paper can provide some insight for future work on this type of PPs.

Through out the paper, the letters {\tt x,z} denote indeterminates. The algebraic closure of a field $F$ is denoted by $\overline F$.

During the preparation of the present paper, the author was informed by M.~Zieve that he found several classes of permutation trinomials of $\Bbb F_{q^2}$ of the form $a\x+b\x^q+\x^{2q-1}$ \cite{Zie-pre}. His approach, which is entirely different from ours, employs linear fractional functions over $\Bbb F_{q^2}$
that map $\mu_{q+1}=\{x\in\Bbb F_{q^2}:x^{q+1}=1\}$ to either $\mu_{q+1}$ or $\Bbb F_q\cup\{\infty\}$. It is a pleasant surprise that the five classes of permutations trinomials in \cite[Theorem~1.1]{Zie-pre}, after a suitable parameterization, comprise precisely the PPs in Theorems~A and B of the present paper. Therefore, \cite{Zie-pre} has provided an alternate proof of the sufficiency of the conditions in Theorems~A and B.

%%%%%%%%%%%%%%%%%%%%%%%%%%%%%%%%%%%%%%%%
%   section 2
%%%%%%%%%%%%%%%%%%%%%%%%%%%%%%%%%%%%%%%%

\section{Theorems A and B}

The results of the paper are the following two theorems which completely determine the permutation trinomials of $\Bbb F_{q^2}$ of the form $a\x+b\x^q+\x^{2q-1}$.

\medskip
\noindent{\bf Theorem A.} {\em Let $f=a{\tt x}+b{\tt x}^q+{\tt x}^{2q-1}\in\Bbb F_{q^2}[{\tt x}]$, where $q$ is odd. Then $f$ is a PP of $\Bbb F_{q^2}$ if and only if one of the following is satisfied.

\begin{itemize}
  \item [(i)] $a=b=0$, $q\equiv 1,3\pmod 6$. 
  \item [(ii)] $(-a)^{\frac{q+1}2}=-1$ or $3$, $b=0$. 
  \item [(iii)] $ab\ne 0$, $a=b^{1-q}$, $1-\frac{4a}{b^2}$ is a square of $\Bbb F_q^*$.
  \item [(iv)] $ab(a-b^{1-q})\ne 0$, $1-\frac{4a}{b^2}$ is a square of $\Bbb F_q^*$, $b^2-a^2b^{q-1}-3a=0$.
\end{itemize} 
}

\noindent{\bf Theorem B.} {\em Let $f=a{\tt x}+b{\tt x}^q+{\tt x}^{2q-1}\in\Bbb F_{q^2}[{\tt x}]$, where $q$ is even. Then $f$ is a PP of $\Bbb F_{q^2}$ if and only if one of the following is satisfied.

\begin{itemize}
  \item [(i)] $a=b=0$, $q=2^{2k}$.
  \item [(ii)] $ab\ne 0$, $a=b^{1-q}$, $\text{\rm Tr}_{q/2}(b^{-1-q})=0$.
  \item [(iii)] $ab(a-b^{1-q})\ne 0$, $\frac a{b^2}\in\Bbb F_q$, $\text{\rm Tr}_{q/2}(\frac a{b^2})=0$, $b^2+a^2b^{q-1}+a=0$.
\end{itemize} 
}

The proofs of Theorems~ A and B, to be given in Sections 3 and 4 separately, have parallel structures. Moreover, for a considerable portion of both proofs, the arguments and computations are valid for all $q$ (even or odd); the conclusions derived in one proof will be allowed to be used in the other.
     
%%%%%%%%%%%%%%%%%%%%%%%%%%%%%%%%%%%%%%%%
%   section 3
%%%%%%%%%%%%%%%%%%%%%%%%%%%%%%%%%%%%%%%%

\section{Proof of Theorem A}

%%%%%%%%%%%%%%%%%%%%%%%%%%%%%%%%%%%%%%%%
\subsection{The case $\boldsymbol{ab(a-b^{1-q})=0}$}\

We first prove Theorem A under the assumption $ab(a-b^{1-q})=0$.

\medskip
{\bf Case 1.} Assume that $a=b=0$. Then $f=\x^{2q-1}$, which is a PP of $\Bbb F_{q^2}$ if and only if $\text{gcd}(2q-1,q^2-1)=1$, i.e., $q\equiv 1,3\pmod 6$.

\medskip
{\bf Case 2.} Assume that $a\ne 0$, $b=0$. 

First assume that $f$ is a PP of $\Bbb F_{q^2}$ and we show that $(-a)^{\frac{q+1}2}=-1$ or $3$. Let $s=\alpha+\beta q\le q^2-2$, where $\alpha,\beta\ge 0$, $\alpha+\beta=q-1$. By \eqref{3.22},
\begin{equation}\label{3.0}
\sum_{x\in\Bbb F_{q^2}}f(x)^s=-a^s\sum_{2i-2j-\alpha=-q,1}\binom\alpha i\binom{q-1-\alpha}ja^{-i-qj}.
\end{equation}
Since $f$ is a PP of $\Bbb F_{q^2}$, by Hermite's criterion, 
\begin{equation}\label{3.1}
\sum_{2i-2j-\alpha=-q,1}\binom\alpha i\binom{q-1-\alpha}ja^{-i-qj}=0,\qquad 0\le \alpha\le q-1.
\end{equation}
Letting $\alpha=1$ in \eqref{3.1}, we have
\begin{align*}
0&=\sum_{-2j=-q+1}\binom{q-2}ja^{-qj}+\sum_{-2j=-q-1,0}\binom{q-2}ja^{-1-qj}\cr
&=\binom{-2}{\frac{q-1}2}a^{-q\frac{q-1}2}+\binom{-2}{\frac{q+1}2}a^{-1-q\frac{q+1}2}+a^{-1}\cr
&=(-1)^{\frac{q-1}2}\frac{q+1}2 a^{\frac{-q^2+q}2}+(-1)^{\frac{q+1}2}\frac{q+3}2 a^{-1-\frac{q^2+q}2}+a^{-1}\cr
&=\frac 12a^{-\frac{q+3}2}\bigl[(-1)^{\frac{q-1}2}a^{\frac{-q^2+1}2}(a^{q+1}-3)+2a^{\frac{q+1}2}\bigr].
\end{align*}
Put $\epsilon=(-1)^{\frac{q-1}2}a^{\frac{-q^2+1}2}\in\{\pm 1\}$. Then we have
\[
a^{q+1}+2\epsilon a^{\frac{q+1}2}-3=0,
\]
which gives $a^{\frac{q+1}2}=\epsilon$ or $-3\epsilon$. Therefore $\epsilon=(-1)^{\frac{q-1}2}a^{\frac{q+1}2(1-q)}=(-1)^{\frac{q-1}2}$. Hence we have $(-a)^{\frac{q+1}2}=-1$ or $3$.

Now assume that $(-1)^{\frac{q+1}2}=-1$ or $3$. To prove that $f$ is a PP of $\Bbb F_{q^2}$, by Hermite's criterion, it suffices to show that $\sum_{x\in\Bbb F_{q^2}}f(x)^s=0$ for all $1\le s\le q^2-2$ and that $0$ is the only root of $f$ in $\Bbb F_{q^2}$. If $s\not\equiv 0\pmod{q-1}$, we have $\sum_{x\in\Bbb F_{q^2}}f(x)^s=0$ by \eqref{3.21}. Assume $s\equiv 0\pmod{q-1}$, i.e., $s=\alpha+\beta q$, $\alpha,\beta\ge 0$, $\alpha+\beta=q-1$. By \eqref{3.0}, $\sum_{x\in\Bbb F_{q^2}}f(x)^s=0$ unless $\alpha$ is odd. So we also assume that $\alpha$ is odd. Put $t=(-a)^{\frac{q+1}2}$. If $t=-1$, then $a^{q+1}=1$. Thus by \eqref{3.0}
\begin{align*}
-a^{-s}\sum_{x\in\Bbb F_{q^2}}f(x)^s&=\sum_{2i-2j-\alpha=-q,1}\binom\alpha i\binom{q-1-\alpha}ja^{-i+j}\kern 1.5cm (a^{-q}=a)\cr
&=\sum_{2(\alpha-i)-2j-\alpha=-q,1}\binom\alpha i\binom{q-1-\alpha}ja^{-(\alpha-i)+j}\cr
&=a^{-\alpha}\sum_{i+j=\frac{\alpha-1}2,\frac{\alpha+q}2}\binom\alpha i\binom{q-1-\alpha}ja^{i+j}\cr
&=a^{-\alpha}\biggl[a^{\frac{\alpha-1}2}\binom{q-1}{\frac{\alpha-1}2}+a^{\frac{\alpha+q}2}\binom{q-1}{\frac{\alpha+q}2}\biggr]\cr
&=a^{-\frac{\alpha+1}2}\bigl[(-1)^{\frac{\alpha-1}2}+a^{\frac{q+1}2}(-1)^{\frac{\alpha+q}2}\bigr]\cr
&=a^{-\frac{\alpha+1}2}\bigl[(-1)^{\frac{\alpha-1}2}-(-1)^{\frac{q+1}2}(-1)^{\frac{\alpha+q}2}\bigr]\cr
&=0.
\end{align*}
If we only assume that $t\in\Bbb F_q$, the above calculation (from the third line) becomes
\begin{align*}
&-a^{-s}\sum_{x\in\Bbb F_{q^2}}f(x)^s\cr
=\,&\sum_{i+j=\frac{\alpha-1}2,\frac{\alpha+q}2}\binom\alpha i\binom{q-1-\alpha}ja^{-\alpha+i-qj}\cr
=\,&\sum_{i+j=\frac{\alpha-1}2,\frac{\alpha+q}2}\binom\alpha i\binom{j+\alpha}\alpha(-1)^ja^{-\alpha+i-qj}\cr
=\,&\sum_i\binom\alpha i\binom{\frac{3\alpha-1}2-i}\alpha(-1)^{\frac{\alpha-1}2-i}a^{-\alpha-\frac{\alpha-1}2q+(q+1)i}\cr
&+\sum_i\binom\alpha i\binom{\frac{3\alpha-1}2-i+\frac{q+1}2}\alpha(-1)^{\frac{\alpha+q}2-i}a^{-\alpha-\frac{\alpha-1}2q-\frac{q+1}2q+(q+1)i}\cr
=\,&\sum_i\binom\alpha i\binom{\frac{3\alpha-1}2-i}\alpha(-1)^{\frac{\alpha-1}2}(-1)^ia^{-\alpha-\frac{\alpha-1}2q-\frac{q+1}2}a^{\frac{q+1}2(2i+1)}\cr
&+\sum_i\binom\alpha i\binom{\frac{3\alpha-1}2-i+\frac{q+1}2}\alpha(-1)^{\frac{\alpha+q}2}(-1)^ia^{-\alpha-\frac{\alpha-1}2q-\frac{q+1}2}a^{\frac{q+1}2 2i}\cr
=\,&(-1)^{\frac{\alpha+q}2}a^{-\alpha-\frac{\alpha-1}2-\frac{q+1}2}\cr
&\cdot\biggl[\sum_i\binom\alpha i\binom{\frac{3\alpha-1}2-i}\alpha(-1)^it^{2i+1}\
+\sum_i\binom\alpha i\binom{\frac{3\alpha-1}2-i+\frac{q+1}2}\alpha(-1)^it^{2i}\biggr].
\end{align*}
When $t=3$, the above sum equals $0$ by \cite[(3.3)]{Hou13}.

\medskip
It remains to show that $0$ is the only root of $f$ in $\Bbb F_{q^2}$. Assume to the contrary that there exists $x\in\Bbb F_{q^2}^*$ such that $f(x)=0$. Then $a=-x^{2(q-1)}$ and hence $(-a)^{\frac{q+1}2}=1\ne-1,3$, which is a contradiction.

\medskip
{\bf Case 3.} Assume that $a=0$, $b\ne 0$. 
In \eqref{3.22} let $\alpha=q-1$ and $\beta=0$. We have 
\[
\sum_{x\in\Bbb F_{q^2}}f(x)^s=-\sum_{k-1=0}\binom{q-1}kb^{q-1-k}=b^{q-2}\ne 0.
\]
So $f$ is not a PP of $\Bbb F_{q^2}$.

\medskip
{\bf Case 4.} Assume that $ab\ne 0$, $a=b^{1-q}$.
In this case, we show that $f$ is a PP of $\Bbb F_{q^2}$ if and only if $1-\frac{4a}{b^2}=1-4b^{-(q+1)}$ is a square of $\Bbb F_q^*$.

First assume that $1-4b^{-(q+1)}$ is a square of $\Bbb F_q^*$. 
To prove that $f$ is a PP of $\Bbb F_{q^2}$, we assume that $f(x)=y$, $x,y\in\Bbb F_{q^2}$, and we show that $x$ uniquely determined by $y$. First consider $y\ne 0$. Note that $y=
b^{1-q}x+bx^q+x^{2q-1}$.
Put $t=b^{-1}xy=b^{-q}x^2+x^{q+1}+b^{-1}x^{2q}\in\Bbb F_q$. Since $x=t\frac by$, $y=f(t\frac by)=tf(\frac by)$.
Thus $t$ is uniquely determined by $y$ and hence $x$ is uniquely determined by $y$. Now consider $y=0$. We claim that $x=0$. Assume to the contrary that there exists $x\ne 0$. Then
\begin{equation}\label{3.3}
x^{2(q-1)}+bx^{q-1}+b^{1-q}=0.
\end{equation}
Thus
\begin{equation}\label{3.4}
x^{q-1}=b\cdot\frac 12\Bigl(-1+\sqrt{1-4b^{-(q+1)}}\,\Bigr),
\end{equation}
where $\sqrt{1-4b^{-(q+1)}}\in\Bbb F_q$. Raising both sides of \eqref{3.4} to the power $q-1$ gives $x^{2(1-q)}=b^{q-1}$. Thus $x^{q-1}=\pm b^{\frac{1-q}2}$ and \eqref{3.3} becomes $2b^{1-q}\pm b^{\frac{3-q}2}=0$, i.e., $b^{\frac{1+q}2}=\mp 2$. Then $1-4b^{-(q+1)}=1-4\cdot\frac 14=0$, which is a contradiction.

Next assume that $f$ is a PP of $\Bbb F_{q^2}$. Put $z=-\frac{a^q}{b^{2q}}$. Letting $\alpha=0$ and $\beta=q-1$ in \eqref{3.22}, we have 
\begin{align*}
0&=\sum_{x\in\Bbb F_{q^2}}f(x)^{(q-1)q}\cr
&=-a^{1-q}\sum_{-j-l=-q}\binom{q-1}j\binom jl a^{-qj}b^{q(j-l)}\cr
&=-a^{1-q}\sum_{1\le l\le \frac q2}\binom{q-1}{q-l}\binom{q-l}la^{-q(q-l)}b^{q(2-2l)}\cr
&=-a^{-q}b^{2q}\sum_{1\le l\le\frac q2}(-1)^{q-l}\binom{-l}la^{ql}b^{-2ql} \numberthis \label{3.4.1}\\
&=-z^{-1}\sum_{1\le l\le\frac q2}\binom{-l}lz^l\cr
&=-z^{-1}\biggl[\,\sum_{0\le l\le\frac q2}\binom{-l}lz^l-1\,\biggr]\cr
&=-z^{-1}\frac 12\bigl[(1+4z)^{\frac{q-1}2}-1\bigr]\kern1cm \text{(by Lemma~\ref{L3.1})}.
\end{align*}
Hence $1+4z$ is a square of $\Bbb F_q^*$. Note that since $z\in\Bbb F_q$, we have $z=-\frac a{b^2}$.

This concludes the proof of Theorem~A under the assumption $ab(a-b^{1-q})=0$.

\medskip
The following lemma, used in the above proof, is a generalization of \cite[Lemma~5.1]{Hou-b} in odd characteristics.  

\begin{lem}\label{L3.1}
Let $q$ be an odd prime power. Then in $\Bbb F_p[\x]$, where $p=\text{\rm char}\,\Bbb F_q$, we have
\begin{equation}\label{3.5}
\sum_{0\le l\le \frac q2}\binom{-l}l\x^l=\frac 12\bigl[1+(1+4\x)^{\frac{q-1}2}\bigr].
\end{equation}
\end{lem}

\begin{proof} For $1\le l\le \frac{q-1}2$, the coefficient of $\x^l$ in the right side of \eqref{3.5} equals (in $\Bbb F_p$)
\begin{align*}
&2^{2l-1}\binom{\frac{q-1}2}l=2^{2l-1}\binom{-\frac 12}l=2^{2l-1}\frac{(-\frac 12)(-\frac 32)\cdots(-\frac{2l-1}2)}{l!}\cr
&=\frac{2^{l-1}(-1)(-3)\cdots(-2l+1)}{l!}=\frac{(-l)(-l-1)\cdots(-2l+1)}{l!}=\binom{-l}l.
\end{align*}
(In the above, we used the formula $2^{l-1}(-1)(-3)\cdots(-2l+1)=(-l)(-l-1)\cdots$ $(-2l+1)$, which is easily proved by induction.)
\end{proof}

\noindent{\bf Remark.} Lemma~\ref{L3.1} can also be derived from Lemma~\ref{4.1}.

%%%%%%%%%%%%%%%%%%%%%%%%%%%%%%%%%%%%
\subsection{The case $\boldsymbol{ab(a-b^{1-q})\ne 0}$, sufficiency}\

Assume that $ab(a-b^{1-q})\ne 0$, $1-\frac{4a}{b^2}$ is a square of $\Bbb F_q^*$, and $b^2-a^2b^{q-1}-3a=0$. We show that $f$ is a PP of $\Bbb F_{q^2}$. (Recall that $q$ is odd.)

\medskip
$1^\circ$ We claim that $a+b+1\ne 0$. Otherwise, $1-\frac{4a}{b^2}=(\frac 2b+1)^2$, which is assumed to be a square of $\Bbb F_q^*$. Thus $b\in\Bbb F_q$ and from $b^2-a^2b^{1-q}-3a=0$, we derive that $a=1=b^{1-q}$, which is a contradiction.

\medskip
$2^\circ$ We claim that for $x\in\Bbb F_{q^2}$, $f(x)\in(a+b+1)\Bbb F_q$ implies $x\in\Bbb F_q$. Assume to the contrary that there exists $x\in\Bbb F_{q^2}\setminus\Bbb F_q$ such that $f(x)/(a+b+1)\in\Bbb F_q$. Then $[f(x)/(a+b+1)]^q=f(x)/(a+b+1)$, i.e., 
\[
(a+b+1)(ax+bx^q+x^{2q-1})^q=(a+b+1)^q(ax+bx^q+x^{2q-1}),
\]
which is equivalent to
\[
\bigl[(a+b+1)^qx^{q-1}+(b+1)^{q+1}-a^{q+1}+(a+b+1)x^{1-q}\bigr](x^q-x)=0.
\]
Since $x^q-x\ne 0$, we have
\begin{equation}\label{3.6}
(a+b+1)^qx^{2(q-1)}+\bigl[(b+1)^{q+1}-a^{q+1}\bigr]x^{q-1}+(a+b+1)=0.
\end{equation}
Let
\begin{equation}\label{3.6.1}
\begin{split}
\Delta\,&=\bigl[(b+1)^{q+1}-a^{q+1}\bigr]^2-4(a+b+1)^{q+1}\cr
&=\bigl[(b^q+1)(b+1)-a^{q+1}\bigr]^2-4(a^q+b^q+1)(a+b+1).
\end{split}
\end{equation}
Since $b^{q-1}=a^{-2}b^2-3a^{-1}$ and $a^{q-1}=b^{2(q-1)}=(a^{-2}b^2-2a^{-1})^2$, we have
\begin{equation}\label{3.6.2}
\begin{split}
\Delta=\,&\bigl[(a^{-2}b^3-3a^{-1}b+1)(b+1)-a^2(a^{-2}b^2-3a^{-1})^2\bigr]^2\cr
&-4\bigl[a(a^{-2}b^2-3a^{-1})^2+a^{-2}b^3-3a^{-1}b+1\bigr](a+b+1)\cr
=\,&\frac{(b^2-4a)(-3a+a^2+ab+b^2)^2}{a^4}\kern1cm\text{(factorization found by computer)}\cr
=\,&\Bigl(1-\frac{4a}{b^2}\Bigr)\Bigl[\frac b{a^2}(-3a+a^2+ab+b^2)\Bigr]^2\cr
=\,&\Bigl(1-\frac{4a}{b^2}\Bigr)\Bigl(b+b\,\frac{b^2-3a}{a^2}+\frac{b^2}a\Bigr)^2\cr
=\,&\Bigl(1-\frac{4a}{b^2}\Bigr)\Bigl(b+b^q+\frac{b^2}a\Bigr)^2,
\end{split}
\end{equation}
which is a square of $\Bbb F_q$. Thus by \eqref{3.6}, $(a+b+1)^qx^{q-1}\in\Bbb F_q$. Therefore
\[
1=\bigl[(a+b+1)^qx^{q-1}\bigr]^{q-1}=(a+b+1)^{1-q}x^{2(1-q)},
\]
and hence
\begin{equation}\label{3.7}
x^{q-1}=\pm(a+b+1)^{\frac{1-q}2}.
\end{equation}
Combining \eqref{3.6} and \eqref{3.7} gives
\[
\pm\bigl[(b+1)^{q-1}-a^{q+1}\bigr](a+b+1)^{\frac{1-q}2}=-2(a+b+1).
\]
Squaring both sides of the above equation, we have
\[
\bigl[(b+1)^{q+1}-a^{q+1}\bigr]^2=4(a+b+1)^{q+1},
\]
which implies $\Delta=0$ by \eqref{3.6.1}. Now the fourth line of \eqref{3.6.2} gives $-3a+a^2+ab+b^2=0$, that is,

\begin{equation}\label{3.8}
-3+a+b+\frac{b^2}a=0.
\end{equation}
Thus $a+b\in\Bbb F_q$. Since $x\notin\Bbb F_q$, we have $x^{q-1}=-1$ by \eqref{3.7}. Now \eqref{3.6} becomes
\begin{equation}\label{3.9}
(b^q+1)(b+1)-a^{q+1}-2(a+b+1)=0.
\end{equation}
Since $\Delta=0$, \eqref{3.6.2} gives
\begin{equation}\label{3.10}
b^q=-b-\frac{b^2}a=-\frac{(a+b)b}a.
\end{equation}
Also,
\begin{equation}\label{3.11}
a^{q+1}=a^2b^{2(q-1)}=a^2b^{-2}\Bigl[-\frac{(a+b)b}a\Bigr]^2=(a+b)^2.
\end{equation}
Thus
\begin{align*}
\text{LHS of \eqref{3.9}}\,&=\Bigl(-b-\frac{b^2}a+1\Bigr)(b+1)-(a+b)^2-2(a+b+1)\cr
&=(a+b+1)\Bigl(-\frac{b^2}a-a-b-1\Bigr),
\end{align*}
and hence
\begin{equation}\label{3.12}
\frac{b^2}a+a+b+1=0.
\end{equation}
Combining \eqref{3.8} and \eqref{3.12} gives $-3=1$, which is a contradiction.

\medskip
$3^\circ$ We now prove that $f$ is a PP of $\Bbb F_{q^2}$. We may assume $b\in\Bbb F_{q^2}\setminus\Bbb F_q$. (If $b\in\Bbb F_q$, then $a\in\Bbb F_q$ and we are done by \cite[Theorem~A]{Hou-a}.) 
Write 
\[
b^2=c+db,\quad c,d\in\Bbb F_q\quad (b+b^q=d,\ b^{q+1}=-c),
\]
and
\[
\frac a{b^2}=e\in\Bbb F_q,\quad\text{i.e.,}\ a=eb^2.
\]
Then $1-4e$ is a square of $\Bbb F_q^*$. We have 
\[
0=b^2-a^2b^{q-1}-3a=b^2-e^2b^{q+3}-3eb^2=b^2(1+e^2c-3e).
\]
Thus $1+e^2c-3e=0$, i.e., 
\begin{equation}\label{3.13}
c=3e^{-1}-e^{-2}.
\end{equation}

Let $z,w\in\Bbb F_{q^2}$ satisfy
\begin{equation}\label{3.14}
f(z)=w.
\end{equation}
We show that $z$ is uniquely determined by $w$. If $w\in(a+b+1)\Bbb F_q$, by $2^\circ$, $z\in\Bbb F_q$. Then \eqref{3.14} gives $z=w/(a+b+1)$. So we assume $w\notin(a+b+1)\Bbb F_q$; it follows that $z\notin\Bbb F_q$. Eq.~\eqref{3.14} is equivalent to 
\begin{equation}\label{3.15}
az^2+bz^{q+1}+z^{2q}=zw.
\end{equation}
Write
\[
\begin{cases}
w=u+vb,\cr
z=x+yb,
\end{cases}
\]
where $u,x\in\Bbb F_q$, $v,y\in\Bbb F_q^*$. We have
\begin{align*}
&az^2+bz^{q+1}+z^{2q}\cr
=\,&e(c+db)(x+yb)^2+b(x+yb)(x+yb^q)+(x+yb^q)^2\cr
=\,&e(c+db)(x+yb)^2+b(x^2+dxy-cy^2)+\bigl[x+y(d-b)\bigr]^2\cr
=\,&\cdots\quad\text{(using the relation $b^2=c+db$)}\cr
=\,&(ce+1)x^2+(2cde+2d)xy+(c^2e+cd^2e+c+d^2)y^2\cr
&+\bigl[(de+1)x^2+(2ce+2d^2e+d-2)xy+(2cde+d^3e-c-d)y^2\bigr]b.
\end{align*}
Also,
\[
zw=(x+yb)(u+vb)=ux+cvy+\bigl[vx+(u+dv)y\bigr]b.
\]
Thus \eqref{3.15} is equivalent to
\begin{equation}\label{3.16}
\begin{cases}
(ce+1)x^2+(2cde+2d)xy+(c^2e+cd^2e+c+d^2)y^2=ux+cvy,\cr
(de+1)x^2+(2ce+2d^2e+d-2)xy+(2cde+d^3e-c-d)y^2=vx+(u+dv)y.
\end{cases}
\end{equation}
From \eqref{3.16} we have
\begin{equation}\label{3.17}
\begin{split}
&\bigl[ (ce+1)x^2+(2cde+2d)xy+(c^2e+cd^2e+c+d^2)y^2\bigr]\bigl[vx+(u+dv)y\bigr]\cr
&-\bigl[(de+1)x^2+(2ce+2d^2e+d-2)xy+(2cde+d^3e-c-d)y^2\bigr](ux+cvy)=0.
\end{split}
\end{equation}
Put $s=x/y$. Then \eqref{3.17} can be written as 
\[
g(s)=0,
\]
where
\begin{equation}\label{3.18}
\begin{split}
g(\x)=\,&(-u-deu+v+cev)\x^3\cr
&+(3u-du-ceu-2d^2eu-cv+3dv+2cdev)\x^2\cr
&+(cu+3du-d^3eu+3cv-cdv+3d^2v-c^2ev+cd^2ev)\x\cr
&+cu+d^2u+c^2eu+cd^2eu+c^2v+2cdv+d^3v-c^2dev\in\Bbb F_q[\x].
\end{split}
\end{equation}
It suffices to show that $g$ has at most one root in $\Bbb F_q$. (If $f(z_1)=f(z_2)$, where $z_1=x_1+y_1b$, $z_2=x_2+y_2b$, $x_1,x_2\in\Bbb F_q$, $y_1,y_2\in\Bbb F_q^*$, then the uniqueness of the root of $g$ in $\Bbb F_q$ implies $\frac{x_1}{y_1}=\frac{x_2}{y_2}$, that is, $z_2=tz_1$ for some $t\in\Bbb F_q^*$. Then $f(z_1)=f(z_2)=f(tz_1)=tf(z_1)$ and hence $t=1$.)

We claim that $-u-deu+v+cev\ne 0$. Otherwise, we have 
\begin{equation}\label{3.18.1}
\begin{split}
&w(a+b+1)^q\cr
=\,&(u+vb)(a^q+b^q+1)\cr
=\,&(u+vb)(eb^{2q}+b^q+1)\cr
=\,&(u+vb)\bigl[e(d-b)^2+d-b+1\bigr]\cr
=\,&u+du+ceu+d^2eu-cv-cdev+(-u-deu+v+cev)b\in\Bbb F_q.
\end{split}
\end{equation}
Then by $2^\circ$, $z\in\Bbb F_q$, which is contrary to our assumption.

Using the relation $c=3e^{-1}-e^{-2}$ (\eqref{3.13}), we find (with computer assistance) that the discriminant of $g$ is given by
\begin{equation}\label{3.19}
D(g)=\frac{(4c+d^2)(1-4e)}{(u+dev-v-cev)^4e^8}\Theta^2,
\end{equation}
where
\begin{equation}\label{3.21.0}
\begin{split}
\Theta=\,&(-2e^2+9e^3-3de^3+9de^4+d^3e^5)u^2\cr
&+(-4e+24e^2-2de^2v-36e^3+9de^3+2d^2e^3-6d^2e^4)uv\cr
&+(-2+15e+de-27e^2-6de^2+9de^3)v^2.
\end{split}
\end{equation}
Moreover, writing $g'=a_2\x^2+a_1\x+a_0$, we have
\begin{equation}\label{3.20}
a_1^2-4a_2a_0=\frac1{e^4}(1+de)\Theta\qquad\text{(another surprise)}.
\end{equation}
In \eqref{3.19}, $1-4e$ is a square of $\Bbb F_q^*$ and $4c+d^2$ is a nonsquare of $\Bbb F_q$ (otherwise $b\in\Bbb F_q$). So if $\Theta\ne0$, then $D(g)$ is a nonsquare of $\Bbb F_q$. It follows that $g$ has at most one root in $\Bbb F_q$.

If $\Theta=0$, then $D(g)=0$. First assume $3\nmid q$. Then $\deg\,g'=2$ and $D(g')=0$. Hence $g'=3\epsilon(\x-r)^2$, where $\epsilon=-u-deu+v+cev\in\Bbb F_q^*$ and $r\in\overline{\Bbb F}_q$. Since $D(g)=0$, we have $\text{gcd}(g,g')\ne 1$, which forces $g=\epsilon(\x-r)^3$. So $g$ has at most one root in $\Bbb F_q$. Next assume $3\mid q$. Then $g'=0\cdot\x^2+a_1\x+a_0$, $a_0,a_1\in\Bbb F_q$. Since $0=a_1^2-4\cdot 0\cdot a_0=a_1^2$, we have $g'=a_0$. Since $D(g)=0$, $g$ has a multiple $r$. Then $(\x-r)^2\mid g$ and hence $\x-r\mid g'$. Since $g'$ is a constant, $g'=0$. It follows that $g=\epsilon\x^3+\gamma$ ($\gamma\in\Bbb F_q$), which has at most one root in $\Bbb F_q$.

This completes the proof of the sufficiency part of Theorem~A under the assumption $ab(a-b^{1-q})\ne 0$.

%%%%%%%%%%%%%%%%%%%%%%%%%%%%%%%%%%%%%%%
\subsection{The case $\boldsymbol{ab(a-b^{1-q})\ne0}$, necessity}\

Let $s=\alpha+\beta q>0$, where $0\le \alpha,\beta\le q-1$. We have
\begin{equation}\label{3.21}
\begin{split}
&\sum_{x\in\Bbb F_{q^2}}f(x)^s\cr
=\,&\sum_{x\in\Bbb F_{q^2}^*}x^s(a+bx^{q-1}+x^{2(q-1)})^{\alpha+\beta q}\cr
=\,&\sum_{x\in\Bbb F_{q^2}^*}x^s(a+bx^{q-1}+x^{2(q-1)})^\alpha(a^q+b^qx^{1-q}+x^{2(1-q)})^\beta\cr
=\,&\sum_{x\in\Bbb F_{q^2}^*}x^s\sum_{i,j,k,l}\binom\alpha i\binom ika^{\alpha-i}(bx^{q-1})^{i-k}(x^{2(q-1)})^k \cr
&\cdot\binom\beta j\binom jl(a^q)^{\beta-j}(b^qx^{1-q})^{j-l}(x^{2(1-q)})^l\cr
=\,&\sum_{i,j,k,l}\binom\alpha i\binom ik\binom\beta j\binom jl a^{\alpha-i+q(\beta-j)}b^{i-k+q(j-l)}\sum_{x\in\Bbb F_{q^2}^*}x^{s+(q-1)(i+k-j-l)}.
\end{split}
\end{equation}
Assume $\alpha+\beta=q-1$. Then \eqref{3.21} becomes
\begin{equation}\label{3.22}
\begin{split}
\sum_{x\in\Bbb F_{q^2}}f(x)^s\,&=\sum_{i,j,k,l}\binom\alpha i\binom ik\binom\beta j\binom jl a^{s-i-qj}b^{i-k+q(j-l)}\sum_{x\in\Bbb F_{q^2}^*}x^{(q-1)(-\alpha-1+i-j-l)}\cr
&=-\sum_{i+k-j-\alpha-1\equiv 0\,(\text{mod}\,q+1)}\binom\alpha i\binom ik\binom\beta j\binom jl a^{s-i-qj}b^{i-k+q(j-l)}\cr
&=-\sum_{i+k-j-\alpha-1=-(q+1),0}\binom\alpha i\binom ik\binom\beta j\binom jl a^{s-i-qj}b^{i-k+q(j-l)}.
\end{split}
\end{equation}
The following lemma holds for all $q$ (even or odd). Eq.~\eqref{3.23} is used here for the proof of Theorem~A; Eq.~\eqref{3.24} will be used in the proof of Theorem~B.

\begin{lem}\label{L3.2}
Assume that $a,b\in\Bbb F_{q^2}^*$ are such that $\x^2+\x+\frac a{b^2}$ has two distinct roots in $\Bbb F_q$. Then
\begin{equation}\label{3.23}
\sum_{x\in\Bbb F_{q^2}}f(x)^{1+(q-2)q}=\frac{2(b^{1-q}-a)(b^2-a^2b^{q-1}-3a)}{a^2(b^2-4a)},
\end{equation}
and for $q>2$,
\begin{equation}\label{3.24}
\begin{split}
&\sum_{x\in\Bbb F_{q^2}}f(x)^{2+(q-3)q}\cr
=\,&\frac{3b^{8-7q}(b^{1-q}-a)(b^2-a^2b^{q-1}-3a)(9a-2b^2+a^3b^{2q-2}-6a^2b^{q-1}+ab^{q+1})}{a^4(b^2-4a)^2}.
\end{split}
\end{equation} 
\end{lem}

\begin{proof}
The computations for \eqref{3.23} and \eqref{3.24} are similar to those in \cite[\S3.3 and Appendix~A]{Hou-a}. Therefore we only give a sketch here. 

Put $z=-a/b^2\in\Bbb F_q$. Letting $\alpha=1$ and $\beta=q-2$ in \eqref{3.22}, we have
\begin{align*}
\sum_{x\in\Bbb F_{q^2}}f(x)^{1+(q-2)q}=-a^{1+(q-2)q}\biggl[&\sum_{-j-l-1=-q}\binom{q-2}j\binom jl a^{-qj}b^{q(j-l)}\cr
&+\sum_{-j-l=-q}\binom{q-2}j\binom jl a^{-1-qj}b^{1+q(j-l)}\cr
&+\sum_{-j-l+1=-q,1}\binom{q-2}j\binom jl a^{-1-qj}b^{q(j-l)}\biggr].
\end{align*}
The right side of the above can be expressed in terms of
\[
\sum_{0\le l\le\frac{q-u}2}\binom{l+v}v\binom{-l}{l+u}z^l,
\]
where $u,v\ge 0$ are small integers. (In this case $0\le u,v\le 1$.) 
We find that 
\begin{equation}\label{3.25}
\begin{split}
\sum_{x\in\Bbb F_{q^2}}f(x)^{1+(q-2)q}=\,&-(a^{q-1}b^{1-q}+a^{-2}b^{1-q})\sum_{0\le l\le\frac{q-1}2}(l+1)\binom{-l}{l+1}z^l\cr
&+2a^{-2}b^{1-q}\sum_{0\le l\le\frac{q-1}2}\binom{-l}{l+1}z^l-a^{-2}b^2\sum_{0\le l\le\frac q2}(l+1)\binom{-l}lz^l\cr
&+2a^{-2}b^2\sum_{0\le l\le\frac q2}\binom{-l}lz^l-a^{-2}b^2-a^{-1}.
\end{split}
\end{equation}
(Cf. \cite[\S3.3]{Hou-a} for the details.) Letting $\alpha=2$ and $\beta=q-3$ ($q>2$) in \eqref{3.22}, in the same way we arrive at
\begin{equation}\label{3.26}
\begin{split}
&\sum_{x\in\Bbb F_{q^2}}f(x)^{2+(q-3)q}=-a^{3-3q}\cr
&\cdot\biggl[(-b^{2q-1}+4a^{-1}b^q-2a^{-2}b-a^{-3}b^3-a^{-4}b^{3-2q})\sum_{0\le l\le\frac q2}\binom{l+2}2\binom{-l}lz^l\cr
&+(2b^{2q-1}-8a^{-1}b^q+6a^{-2}b+3a^{-3}b^3-a^{-4}b^{3-2q})\sum_{0\le l\le\frac q2}(l+1)\binom{-l}l z^l\cr
&+(-b^{2q-1}+4a^{-1}b^q-6a^{-2}b-3a^{-3}b^3+5a^{-4}b^{3-2q})\sum_{0\le l\le\frac q2}\binom{-l}l z^l\cr
&-2a^{-3}b^{2-q}\sum_{0\le l\le\frac {q-1}2}\binom{l+2}2\binom{-l}{l+1} z^l \cr
&+(6a^{-3}b^{2-q}-3a^{-4}b^{3-2q})\sum_{0\le l\le \frac{q-1}2}(l+1)\binom{-l}{l+1} z^l\cr
&+(-6a^{-3}b^{2-q}+6a^{-4}b^{3-2q})\sum_{0\le l\le\frac{q-1}2}\binom{-l}{l+1} z^l\cr
&+4a^{-2}b+a^{-3}b^3-3a^{-3}b^{2-q}-3a^{-4}b^{3-2q}\biggr].
\end{split}
\end{equation}
(Cf. \cite[Appendix~A]{Hou-a} for the details.) It is known that
\begin{align}
&\sum_{0\le l\le\frac q2}\binom{-l}l z^l=1 &\text{(\cite[Lemma~5.1]{Hou-b})},\label{3.27} \\ 
&\sum_{0\le l\le\frac q2}(l+1)\binom{-l}l z^l=\frac{1+3z}{1+4z}&\text{(\cite[Lemma~3.2]{Hou-a})},\label{3.28} \\
&\sum_{0\le l\le\frac q2}\binom{l+2}2\binom{-l}lz^l=\frac{1+6z+11z^2}{(1+4z)^2},\quad q>2&\text{(\cite[Lemma~4.2]{Hou-a})},\label{3.29} \\
&\sum_{0\le l\le\frac{q-1}2}\binom{-l}{l+1} z^l=1&\text{(\cite[Lemma~3.3]{Hou-a})},\label{3.30} \\
&\sum_{0\le l\le \frac{q-1}2}(l+1)\binom{-l}{l+1} z^l=\frac{2z}{1+4z}&\text{(\cite[Lemma~3.3]{Hou-a})},\label{3.31} \\
&\sum_{0\le l\le\frac {q-1}2}\binom{l+2}2\binom{-l}{l+1} z^l =\frac{3z(1+2z)}{(1+4z)^2},\quad q>2 &\text{(\cite[Lemma~A.1]{Hou-a})}.\label{3.32}
\end{align}
Making the substitutions \eqref{3.27} -- \eqref{3.32} in \eqref{3.25} and \eqref{3.26} produces \eqref{3.23} and \eqref{3.24}; again, confer \cite[\S3.3 and Appendix~A]{Hou-a} for the details.
\end{proof}

Now assume that $ab(a-b^{1-q})\ne 0$ and $f$ is a PP of $\Bbb F_{q^2}$. We show that $1-\frac{4a}{b^2}$ is a nonsquare of $\Bbb F_q^*$ and $b^2-a^2b^{q-1}-3a=0$. That $1-\frac{4a}{b^2}$ is a square of $\Bbb F_q^*$ follows from \eqref{3.4.1}. Then $\x^2+\x+\frac a{b^2}$ has two distinct roots in $\Bbb F_q$. It follows from \eqref{3.23} that $b^2-a^2b^{q-1}-3a=0$.

%%%%%%%%%%%%%%%%%%%%%%%%%%%%%%%%%%%%%%%
%  section 4
%%%%%%%%%%%%%%%%%%%%%%%%%%%%%%%%%%%%%%%

\section{Proof of Theorem~B}

%%%%%%%%%%%%%%%%%%%%%%%%%%%%%%%%%%%%%%%

\subsection{The case $\boldsymbol{ab(a-b^{1-q})=0}$}\

We first prove Theorem~B under the assumption $ab(a-b^{1-q})=0$. 

\medskip
{\bf Case 1.} Assume that $a=b=0$. Then $f=\x^{2q-1}$, which is a PP of $\Bbb F_{q^2}$ if and only if $q=2^{2k}$.

\medskip
{\bf Case 2.} Assume that $a\ne 0$, $b=0$. By \eqref{3.0} with $s=0+(q-1)q$, we have
\[
\sum_{x\in\Bbb F_{q^2}}f(x)^{(q-1)q}=a^s\sum_{-2j=-q}\binom{q-1}ja^{-qj}=a^s\binom{q-1}{\frac q2}a^{-\frac{q^2}2}=a^{\frac{q^2}2-q}\ne 0.
\]
Therefore $f$ is not a PP of $\Bbb F_{q^2}$.

\medskip
{\bf Case 3.} Assume that $a=0$, $b\ne 0$. Let $\alpha=q-1$ and $\beta=0$ in \eqref{3.22}, we have 
\[
\sum_{x\in\Bbb F_{q^2}}f(x)^{q-1}=\sum_{k-1=0}\binom{q-1}kb^{q-1-k}=b^{q-2}\ne 0.
\]
So $f$ is not a PP of $\Bbb F_{q^2}$.

\medskip
{\bf Case 4.} Assume that $ab\ne 0$, $a=b^{1-q}$. Then $\frac a{b^2}=\frac{b^{1-q}}{b^2}=b^{-1-q}$. We show that $f$ is a PP of $\Bbb F_{q^2}$ if and only if $\text{Tr}_{q/2}(b^{-1-q})=0$. The ``only if'' part is proved in Section~4.3 after Lemma~\ref{L4.1}. Now assume $\text{Tr}_{q/2}(b^{-1-q})=0$. Let $x,y\in\Bbb F_{q^2}$ satisfying $f(x)=y$. We show that $x$ is uniquely determined by $y$.

First assume $y\ne 0$. We have $y=f(x)=b^{1-q}x+bx^q+x^{2q-1}$. Let $t=b^{-1}xy=b^{-q}x^2+x^{q+1}+b^{-1}x^{2q}\in\Bbb F_q$. Since $x=t\frac by$, we have $tf(\frac by)=y$. Hence $t$ is uniquely determined by $y$ and so is $x$.

Next assume $y=0$. We claim that $x=0$. Otherwise we have
\begin{equation}\label{4.1}
x^{2(q-1)}+bx^{q-1}+b^{1-q}=0,
\end{equation}
i.e.,
\[
\Bigl(\frac{x^{q-1}}b\Bigr)^2+\frac{x^{q-1}}b+b^{-1-q}=0.
\]
Since $\text{Tr}_{q/2}(b^{-1-q})=0$, we have $\frac{x^{q-1}}b\in\Bbb F_q$. Then $(\frac{x^{q-1}}b)^{q-1}=1$, which gives $x^{2(q-1)}=b^{1-q}$. Then \eqref{4.1} implies $bx^{1-q}=0$, which is a contradiction.

%%%%%%%%%%%%%%%%%%%%%%%%%%%%%%%%%%%%%%%
\subsection{The case $\boldsymbol{ab(a-b^{1-q})\ne 0}$, sufficiency}\

Assume that $ab(a-b^{1-q})\ne 0$, $\frac a{b^2}\in\Bbb F_q$, $\text{Tr}_{q/2}(\frac a{b^2})=0$, and $b^2+a^2b^{q-1}+a=0$. We show that $f$ is a PP of $\Bbb F_{q^2}$.

\medskip
$1^\circ$ We claim that $a+b+1\ne0$. Otherwise, we have
\[
\Bigl(\frac 1b\Bigr)^2+\frac 1b+\frac a{b^2}=0.
\]
Since $\frac a{b^2}\in\Bbb F_q$ and $\text{Tr}_{q/2}(\frac a{b^2})=0$, we have $\frac 1b\in\Bbb F_q$. Then $b^2+a^2b^{q-1}+a=0$ gives $a=1=b^{1-q}$, which is a contradiction. 

\medskip
$2^\circ$ We may assume $b\in\Bbb F_{q^2}\setminus\Bbb F_q$. (If $b\in\Bbb F_q$, then $a\in\Bbb F_q$ and we are done by \cite[Theorem~B]{Hou-a}.) Write
\[
b^2=c+db,
\]
where $d=b+b^q\in\Bbb F_q$ and $c=b^{q+1}\in\Bbb F_q$, and put $e=\frac a{b^2}\in\Bbb F_q$. Then $b^2+a^2b^{q-1}+a=0$ gives
\begin{equation}\label{4.2}
c=e^{-1}+e^{-2}.
\end{equation}

\medskip
$3^\circ$ We claim that for $z\in\Bbb F_{q^2}$, $f(z)\in (a+b+1)\Bbb F_q$ implies $z\in\Bbb F_q$. Assume to the contrary that there exists $z\in\Bbb F_{q^2}\setminus\Bbb F_q$ such that $f(z)=w\in (a+b+1)\Bbb F_q$. Write $z=x+yb$ and $w=u+vb$, where $x,y,u,v\in\Bbb F_q$, $y\ne 0$. Then by \eqref{3.18.1}, which also holds for even $q$, we have 
\begin{equation}\label{4.3}
(de+1)u=(ce+1)v,
\end{equation}
which by \eqref{4.2} can be written as
\[
v=e(de+1)u.
\]

First assume $u\ne 0$. The relation \eqref{4.3} allows us to eliminate the $x^2$ and $x$ terms in \eqref{3.16}; the result is
\begin{equation}\label{4.4}
v(c^2e+cd^2e+c+d^2)y^2+udxy+u(d^3e+c+d)y^2=(cv^2+u^2+duv)y.
\end{equation}
Use the relations $v=e(de+1)u$ and $c=e^{-1}+e^{-2}$ in \eqref{4.4} and cancel $uy$. We have
\begin{equation}\label{4.5}
x=(d+e^{-1})y+d^{-1}e(1+d+d^2e^2)u.
\end{equation}
Making the substitution \eqref{4.5} in the first equation of \eqref{3.16} gives
\begin{equation}\label{4.6}
y^2+e^2(de+1)uy+\frac{e^3u^2}{d^2}\bigl[d(de+1)^2+(de+1)^4\bigr]=0.
\end{equation}
However,
\begin{align*}
&\text{Tr}_{q/2}\Bigl[\frac{e^3u^2[d(de+1)^2+(de+1)^4]}{d^2[e^2(de+1)u]^2}\Bigr]\cr
=\,&\text{Tr}_{q/2}\Bigl[\frac{d+d^2e^2+1}{d^2e}\Bigr]= \text{Tr}_{q/2}\Bigl[e+\frac 1{de}+\frac 1{d^2e}\Bigr]\cr
=\,&\text{Tr}_{q/2}\Bigl[\frac 1{d^2e^2}+\frac 1{d^2e}\Bigr]=\text{Tr}_{q/2}\Bigl(\frac{e^{-2}+e^{-1}}{d^2}\Bigr)=\text{Tr}_{q/2}\Bigl(\frac c{d^2}\Bigr)\cr
=\,&1 \kern 1cm\text{(since $b\notin\Bbb F_q$)}.
\end{align*}
Thus by \eqref{4.6}, $y\notin\Bbb F_q$, which is a contradiction.

Next assume $u=0$. We have
\begin{equation}\label{4.7}
a+bz^{q-1}+z^{2(q-1)}=0,
\end{equation}
i.e.,
\[
\Bigl(\frac{z^{q-1}}b\Bigr)^2+\frac{z^{q-1}}b+\frac a{b^2}=0.
\]
Since $\text{Tr}_{q/2}(\frac a{b^2})=0$, we have $\frac{z^{q-1}}b\in\Bbb F_q$. Then $(\frac{z^{q-1}}b)^{q-1}=1$, which gives $z^{2(1-q)}=b^{q-1}$. By \eqref{4.7}, $a^2+b^2z^{2(q-1)}+z^{4(q-1)}=0$, i.e.,
\[
a^2z^{4(1-q)}+b^2z^{2(1-q)}+1=0.
\]
Since $z^{2(1-q)}=b^{q-1}=\frac{b^2+a}{a^2}$, we have 
\[
a^2\Bigl(\frac{b^2+a}{a^2}\Bigr)^2+b^2\cdot\frac{b^2+a}{a^2}+1=0,
\] which gives $\frac{b^2}a=0$, a contradiction.

\medskip
$4^\circ$ Let $z,w\in\Bbb F_{q^2}$ satisfy $f(z)=w$. We show that $z$ is uniquely determined by $w$. Write $z=x+yb$ and $w=u+vb$, where $x,y,u,v\in\Bbb F_q$ satisfy \eqref{3.16}.

If $u(de+1)=v(ce+1)$, by \eqref{3.18.1}, $w\in(a+b+1)\Bbb F_q$. By $3^\circ$, $z\in\Bbb F_q$, and hence $z=w/(a+b+1)$, which is unique. Therefore we assume $u(de+1)\ne v(ce+1)$. It follows from \eqref{3.18.1} that $w\notin (a+b+1)\Bbb F_q$ and hence $z\notin\Bbb F_q$. The polynomial $g$ in \eqref{3.18} is cubic. It suffices to show that $g$ has at most one root in $\Bbb F_q$. (See the explanation following \eqref{3.18}.)

\medskip
{\bf Case 1.} Assume to the contrary that $g$ has three distinct roots in $\Bbb F_q$. Let $\Bbb Q_2$ denote the field of $2$-adic numbers. Let $F/\Bbb Q_2$ be the unramified extension with residue field $\Bbb F_q$ and let $\frak o$ be the ring of integers of $F$. Lift $d,e,u,v\in\Bbb F_q$ to $\widetilde d, \widetilde e, \widetilde u, \widetilde v\in\frak o$ and put $\widetilde c=3\widetilde e\,^{-1}-\widetilde e\,^{-2}$. Then $\widetilde c\in\frak o$ since $e\ne 0$.
 Let $\widetilde g\in\frak o[\x]$ be the resulting lift of $g$. By Hensel's lemma \cite[Theorem~4.4.2]{Koc}, $\widetilde g$ splits with distinct roots in $\frak o$. Thus the discriminant $D(\widetilde g)$ of $\widetilde g$ is a square of $\frak o\setminus\{0\}$.  The calculation of \eqref{3.19}, which does not depend on the field, still holds for $D(\widetilde g)$. Thus we have
\begin{equation}\label{4.8}
D(\widetilde g)=\frac{(4\widetilde c+\widetilde d\,^2)(1-4\widetilde e)}{(\widetilde u+\widetilde d\,\widetilde e\,\widetilde u-\widetilde v-\widetilde c\,\widetilde e\,\widetilde v)^4\widetilde e\,^8}\widetilde\Theta^2,
\end{equation}
where $\widetilde\Theta$ is given by \eqref{3.21.0} with $d,e,u,v$ replaced by $\widetilde d, \widetilde e, \widetilde u, \widetilde v$, respectively. We claim that $1-4\widetilde e$ is a square of $\frak o\setminus\{0\}$. Since $\x^2+\x+e\in\Bbb F_q[\x]$ is reducible with distinct roots, $\x^2+\x+\widetilde e\in\frak o[\x]$ is reducible with distinct roots in $\frak o$ (Hensel's lemma). Hence $1-4\widetilde e$ is a square of $\frak o\setminus\{0\}$. It follows from \eqref{4.8} that $4\widetilde c+\widetilde d^2$ is a square of $\frak o\setminus\{0\}$. Therefore $\x^2-\widetilde d\x-\widetilde c\in F[\x]$ is reducible; its roots are in $\frak o$ since $\frak o$ is integrally closed in $F$. Thus $\x^2-d\x-c\in\Bbb F_q[\x]$ is reducible, which implies $b\in\Bbb F_q$, a contradiction.

\medskip
{\bf Case 2.} Assume that $g$ has a multiple root in $\overline{\Bbb F}_q$. Then $D(g)=0$, which implies $\Theta=0$ by \eqref{3.19}. We have $\deg\,g'=2$ and by \eqref{3.20}, $D(g')=0$. Therefore $g=\epsilon(\x-r)^3$ for some $\epsilon\in\Bbb F_q^*$ and $r\in \overline{\Bbb F}_q$. So $g$ has at most one root in $\Bbb F_q$.  

%%%%%%%%%%%%%%%%%%%%%%%%%%%%%%%%%%%%%%%%
\subsection{The case $\boldsymbol{ab(a-b^{1-q})\ne 0}$, necessity}\

\begin{lem}\label{L4.1}
Let $q$ be any prime power. Let ${\tt z}$ be a transcendental over $\Bbb F_p$ ($p=\text{\rm char}\,\Bbb F_q$) and write $\x^2+\x-{\tt z}=(\x-r_1)(\x-r_2)$, $r_1,r_2\in\overline{\Bbb F_p({\tt z})}$. Then we have 
\[
\sum_{0\le l\le\frac q2}\binom{-l}l{\tt z}^l=-\frac{r_1^{q+1}-r_2^{q+1}}{r_1-r_2}.
\]
\end{lem}

\begin{proof}We denote the constant term of a Laurent series in $\x$ by $\text{ct}(\ )$. By the computation in the proof of \cite[Lemma~5.1]{Hou-b}, we have
\begin{align*}
\sum_{0\le l\le\frac q2}\binom{-l}l{\tt z}^l\,&=\text{ct}\Bigl[\frac{-{\tt z}^{q+1}}{\x^q}\frac 1{(\x-r_1)(\x-r_2)}\Bigr]\cr
&=\text{ct}\Bigl[\frac{{\tt z}^{q+1}}{\x^q}\frac 1{r_1-r_2}\Bigl(\frac 1{r_1}\,\frac 1{1-\frac{\x}{r_1}}-\frac 1{r_2}\,\frac 1{1-\frac{\x}{r_2}}\Bigr)\Bigr]\cr
&=\frac{{\tt z}^{q+1}}{r_1-r_2}\Bigl(\frac 1{r_1^{q+1}}-\frac 1{r_2^{q+1}}\Bigr)\cr
&=\frac{-{\tt z}^{q+1}}{(r_1r_2)^{q+1}}\cdot\frac{r_1^{q+1}-r_2^{q+1}}{r_1-r_2}\cr
&=-\frac{r_1^{q+1}-r_2^{q+1}}{r_1-r_2}.
\end{align*}
\end{proof}

Assume that $ab(a-b^{1-q})\ne 0$ and that $f$ is a PP of $\Bbb F_{q^2}$. We show that $\frac a{b^2}\in\Bbb F_q$, $\text{Tr}_{q/2}(\frac a{b^2})=0$ and $b^2+a^2b^{q-1}+a=0$. 

Let $z=\frac{a^q}{b^{2q}}$. By \eqref{3.4.1} we have
\begin{equation}\label{4.9}
\sum_{0\le l\le\frac q2}\binom{-l}lz^l=1.
\end{equation}
Write $\x^2+\x+z=(\x+r_1)(\x+r_2)$, $r_1,r_2\in\overline{\Bbb F}_q$. It follows from \eqref{4.9} and Lemma~\ref{L4.1} that 
\[
\frac{r_1^{q+1}+r_2^{q+1}}{r_1+r_2}=1.
\]
Since $r_1+r_2=1$, we have
\[
1=r_1^{q+1}+r_2^{q+1}=r_1^{q+1}+(r_1+1)^{q+1}=r_1+r_1^q+1.
\]
So $r_1+r_1^q=0$, i.e., $r_1\in\Bbb F_q$. It follows that $r_2\in\Bbb F_q$ and $z\in\Bbb F_q$. Then $z=\frac{a^q}{b^{2q}}=\frac a{b^2}$. Since $\x^2+\x+z\in\Bbb F_q[\x]$ is reducible, we have $\text{Tr}_{q/2}(z)=0$.

By \eqref{3.24}, we have
\[
(b^2+a^2b^{q-1}+a)(1+a^2b^{2q-2}+b^{q+1})=0.
\]
It remains to show that $1+a^2b^{2q-2}+b^{q+1}\ne 0$. Assume to the contrary that $1+a^2b^{2q-2}+b^{q+1}=0$. Then we have
\[
1+ab^{q-1}+b^{\frac{q+1}2}=0,
\]
i.e.,
\[
a+b\cdot b^{\frac 12(1-q)}+b^{1-q}=0.
\]
(Note that $b^{\frac 12}\in\Bbb F_{q^2}$ is well defined since $q$ is even.) Then $f(b^{-\frac 12})=0$, which is a contradiction.

The proof of Theorem~B is now complete.  

%%%%%%%%%%%%%%%%%%%%%%%%%%%%%%%%%%%%%%%%

\end{document}